\documentclass[11pt]{amsart}%
\usepackage{graphicx}
\usepackage{amscd}
\usepackage{amsmath}
\usepackage{amsfonts}
\usepackage{amssymb}%
\providecommand{\U}[1]{\protect\rule{.1in}{.1in}}
\theoremstyle{plain}

\newtheorem{definition}{Definition}[section]

\newtheorem{theorem}{Theorem}
\theoremstyle{definition}

\newtheorem{example}{Example}[section]

\numberwithin{equation}{section}
\numberwithin{theorem}{section}
\begin{document}
\title[Jacobi-Stirling\ Numbers]{The Jacobi-Stirling Numbers}
\author{George E. Andrews}
\address{Department of Mathematics, The Pennsylvania State University, University Park,
PA., 16801}
\email{andrews@math.psu.edu}
\author{Eric S. Egge}
\address{Department of Mathematics, Carleton College, Northfield, MN 55057}
\email{eegge@carleton.edu}
\author{Wolfgang Gawronski}
\address{Department of Mathematics, University of Trier, 54286 Trier, Germany}
\email{gawron@uni-trier.de}
\author{Lance L. Littlejohn}
\address{Department of Mathematics, Baylor University, One Bear Place \#97328, Waco, TX 76798-7328}
\email{Lance\_Littlejohn@baylor.edu}
\date{December 19, 2011}

\begin{abstract}
The Jacobi-Stirling numbers were discovered as a result of a problem involving
the spectral theory of powers of the classical second-order Jacobi
differential expression. Specifically, these numbers are the coefficients of
integral composite powers of the Jacobi expression in Lagrangian symmetric
form. Quite remarkably, they share many properties with the classical Stirling
numbers of the second kind which, as shown in \cite{LW}, are the coefficients
of integral powers of the Laguerre differential expression. In this paper, we
establish several properties of the Jacobi-Stirling numbers and its companions
including combinatorial interpretations thereby extending and supplementing
known contributions to the literature \cite{Andrews-Littlejohn},
\cite{Andrews-Gawronski-Littlejohn}, \cite{Egge}, \cite{Gelineau-Zeng}, and
\cite{Mongelli2}.

\end{abstract}
\maketitle

\section{Introduction}

The Jacobi-Stirling numbers, defined for $n,j\in\mathbb{N}$ by
\begin{equation}%
%TCIMACRO{\QATOPD{\{}{\}}{n}{j}}%
%BeginExpansion
\genfrac{\{}{\}}{0pt}{}{n}{j}%
%EndExpansion
_{\alpha,\beta}:=\sum_{r=0}^{j}(-1)^{r+j}\frac{(\alpha+\beta+2r+1)\Gamma
(\alpha+\beta+r+1)\left[  r(r+\alpha+\beta+1)\right]  ^{n}}{r!(j-r)!\Gamma
(\alpha+\beta+j+r+2)}, \label{Jacobi-Stirling number}%
\end{equation}
were discovered in 2007 in the course of the left-definite operator-theoretic
study \cite{EKLWY} of the classical second-order Jacobi differential
expression%
\begin{equation}
\ell_{\alpha,\beta}[y](x)=\lambda y(x)\quad(x\in(-1,1)),
\label{JacobiDEwithlambda}%
\end{equation}
where%

\begin{align}
\ell_{\alpha,\beta}[y](x)  &  :=\dfrac{1}{w_{\alpha,\beta}(x)}\left(  \left(
-(1-x)^{\alpha+1}(1+x)^{\beta+1})y^{\prime}(x)\right)  ^{\prime}%
+k(1-x)^{\alpha}(1+x)^{\beta}y(x)\right) \label{JacobiDE}\\
&  =-(1-x^{2})y^{\prime\prime}(x)+(\alpha-\beta+(\alpha+\beta+2)x)y^{\prime
}(x)+ky(x)\quad(x\in(-1,1)).\nonumber
\end{align}
Here, we assume that $\alpha,\beta>-1$, $k$ is a fixed, non-negative constant,
and $w_{\alpha,\beta}(x)$ is the Jacobi weight function defined by
\begin{equation}
w_{\alpha,\beta}(x):=(1-x)^{\alpha}(1+x)^{\beta}\quad(x\in(-1,1)).
\label{Jacobi weight function}%
\end{equation}

The Jacobi-Stirling numbers are the main subject of this paper; as we will
see, these numbers are similar in many ways to the classical Stirling numbers
of the second kind. Moreover, this manuscript may be viewed as a continuation
of the combinatorial results obtained in \cite{Andrews-Littlejohn},
\cite{Andrews-Gawronski-Littlejohn}, and \cite{Egge}; each of these papers
deals exclusively with various properties of the Legendre-Stirling numbers
$\{PS_{n}^{(j)}\}$, a special case of the Jacobi-Stirling numbers. Indeed, by
definition, $PS_{n}^{(j)}\hspace{-4pt}=%
%TCIMACRO{\QATOPD{\{}{\}}{n}{j}}%
%BeginExpansion
\genfrac{\{}{\}}{0pt}{}{n}{j}%
%EndExpansion
_{0,0}.$ The Jacobi-Stirling numbers have generated a significant amount of
interest from other researchers in combinatorics. In this respect, we note
that the Legendre-Stirling numbers appear in some recent work \cite{CKRT}
related to the Boolean number of a Ferrers graph; these authors also show that
there is an interesting connection between Legendre-Stirling numbers and the
Genocchi numbers of the second kind. Our present paper can also be contrasted
with recent work of Gelineau and Zeng \cite{Gelineau-Zeng}, who present an
alternative approach to the combinatorics of the Jacobi-Stirling numbers. In a
parallel development, Mongelli has recently established the total positivity
of the Jacobi-Stirling numbers in \cite{Mongelli1}. In addition, in the recent
manuscript \cite{Mongelli2} Mongelli shows that the Jacobi-Stirling numbers
are specializations of the elementary and complete homogeneous symmetric
functions; he also obtains combinatorial interpretations of a wide class of
numbers which include the Jacobi-Stirling numbers as special cases. We remark
that Comtet \cite{Comtet-1972} further generalized the classical Stirling
numbers in \cite{Comtet-1972}. For an excellent account of Stirling numbers of
the first and second kind, see Comtet's text \cite[Chapter V]{Comtet}.

The contents of this paper are as follows. In Section \ref{Background}, we
briefly review the Jacobi-Stirling numbers from the original context of
left-definite theory. Section \ref{Comparison} deals with a comparison of
various properties of the classical Stirling numbers of the second kind and
the Jacobi-Stirling numbers. The proofs of many of these properties are
similar to the proofs given in \cite{Andrews-Gawronski-Littlejohn} so proofs,
in this paper, will either be brief or omitted completely. In Section
\ref{Main Results1} we give a combinatorial interpretation of the
Jacobi-Stirling numbers in terms of set partitions with a prescribed number of
blocks. In Section \ref{JS1} we study the Jacobi-Stirling numbers of the first
kind and prove several properties of these numbers which are analogues of
properties of the classical Stirling numbers of the first kind. In particular,
we prove a reciprocity result connecting the Jacobi-Stirling numbers and the
Jacobi-Stirling numbers of the first kind. In Section \ref{sec:JS1combin} we
give two combinatorial interpretations of the Jacobi-Stirling numbers of the
first kind in terms of ordered pairs of permutations with prescribed numbers
of cycles.

\textbf{Notation.} It is clear from (\ref{Jacobi-Stirling number}) that $%
%TCIMACRO{\QATOPD{\{}{\}}{n}{j}}%
%BeginExpansion
\genfrac{\{}{\}}{0pt}{}{n}{j}%
%EndExpansion
_{\alpha,\beta}$ is a function of $\alpha+\beta+1$, rather than of $\alpha$
and $\beta$ individually. With this in mind, we might wish to follow Gelineau
and Zeng in setting $z=\alpha+\beta+1$. It turns out, however, that it is more
natural for our investigations to view these numbers as a function of $\gamma$
where $\gamma=\dfrac{z+1}{2}.$ With this in mind, we write $%
%TCIMACRO{\QATOPD{\{}{\}}{n}{j}}%
%BeginExpansion
\genfrac{\{}{\}}{0pt}{}{n}{j}%
%EndExpansion
_{\gamma}$ to denote the Jacobi-Stirling number $%
%TCIMACRO{\QATOPD{\{}{\}}{n}{j}}%
%BeginExpansion
\genfrac{\{}{\}}{0pt}{}{n}{j}%
%EndExpansion
_{\alpha,\beta}$, where $2\gamma-1=\alpha+\beta+1$. With this notational
change, we see from (\ref{Jacobi-Stirling number}) that
\begin{equation}%
%TCIMACRO{\QATOPD{\{}{\}}{n}{j}}%
%BeginExpansion
\genfrac{\{}{\}}{0pt}{}{n}{j}%
%EndExpansion
_{\gamma}:=\sum_{r=0}^{j}(-1)^{r+j}\frac{(2r+2\gamma-1)\Gamma(r+2\gamma
-1)\left[  r(r+2\gamma-1)\right]  ^{n}}{r!(j-r)!\Gamma(j+r+2\gamma)}.
\label{Jacobi-Stirling New Notation}%
\end{equation}
Observe that the numbers $%
%TCIMACRO{\QATOPD{\{}{\}}{n}{j}}%
%BeginExpansion
\genfrac{\{}{\}}{0pt}{}{n}{j}%
%EndExpansion
_{1}$ (that is, $\gamma=1)$ are precisely the Legendre-Stirling numbers. The
following table lists some Jacobi-Stirling numbers for small values of $n$ and
$j$.%

\[%
\begin{tabular}
[c]{|l|l|l|l|l|l|l|}\hline
$j/n$ & $n=0$ & $n=1$ & $n=2$ & $n=3$ & $n=4$ & $n=5$\\\hline
$j=0$ & $1$ & $0$ & $0$ & $0$ & $0$ & $0$\\\hline
$j=1$ & $0$ & $1$ & $2\gamma$ & $4\gamma^{2}$ & $8\gamma^{3}$ & $16\gamma^{4}%
$\\\hline
$j=2$ & $0$ & $0$ & $1$ & $6\gamma+2$ & $28\gamma^{2}+20\gamma+4$ &
$120\gamma^{3}+136\gamma^{2}+56\gamma+8$\\\hline
$j=3$ & $0$ & $0$ & $0$ & $1$ & $12\gamma+8$ & $100\gamma^{2}+140\gamma
+52$\\\hline
$j=4$ & $0$ & $0$ & $0$ & $0$ & $1$ & $20\gamma+20$\\\hline
$j=5$ & $0$ & $0$ & $0$ & $0$ & $0$ & $1$\\\hline
\end{tabular}
\ \ \
\]

\begin{center}
\textbf{Table 1: Jacobi-Stirling Numbers \medskip}
\end{center}

A similar table, with the Jacobi-Stirling numbers in terms of $\alpha$ and
$\beta,$ can be found in \cite{EKLWY}; also, in \cite{Gelineau-Zeng}, the
authors have a table of Jacobi-Stirling numbers given in terms of $z.$ We show
below, in Theorem \ref{Properties of the Jacobi-Stirling numbers}(iii), that
the Jacobi-Stirling numbers satisfy a certain triangular recurrence relation
that allows for a fast computation of these numbers.

As with the classical Stirling numbers there occur, in a natural way,
corresponding (signless) Jacobi-Stirling numbers $%
%TCIMACRO{\QATOPD{[}{]}{n}{j}}%
%BeginExpansion
\genfrac{[}{]}{0pt}{}{n}{j}%
%EndExpansion
_{\gamma}$ of the first kind (see Section \ref{JS1} below; see also
\cite{Egge}, \cite{Gelineau-Zeng}, and \cite{Mongelli2}). In view of this, we
occasionally will call $%
%TCIMACRO{\QATOPD{\{}{\}}{n}{j}}%
%BeginExpansion
\genfrac{\{}{\}}{0pt}{}{n}{j}%
%EndExpansion
_{\gamma}$ \textit{Jacobi-Stirling numbers of the second kind}.

\section{Background\label{Background}}

We give a brief account in this section on the origin of the Jacobi-Stirling
numbers. They were discovered in a certain spectral study of the Jacobi
differential expression (\ref{JacobiDE}); for specific details see
\cite{EKLWY}.

When%
\[
\lambda=\lambda_{r}=r(r+\alpha+\beta+1)+k\quad(r\in\mathbb{N}_{0}),
\]
one classic solution of the Jacobi differential equation
(\ref{JacobiDEwithlambda}) is
\begin{equation}
y(x)=P_{r}^{(\alpha,\beta)}(x)=\binom{r+\alpha}{r}F(-r,1+\alpha+\beta
+r;1+\alpha;\frac{1-x}{2})\quad(r\in\mathbb{N}_{0}), \label{Jacobi polynomial}%
\end{equation}
where $P_{r}^{(\alpha,\beta)}(x)$ is the Jacobi polynomial of degree $r$ as
defined in the classic text of Szeg\"{o} \cite[Chapter IV]{Szego}.

The most natural setting for an analytic study of (\ref{JacobiDEwithlambda})
is the Hilbert function space
\[
L^{2}((-1,1);w_{\alpha,\beta}(x)):=L_{\alpha,\beta}^{2}(-1,1),
\]
defined by
\begin{equation}
L_{\alpha,\beta}^{2}(-1,1):=\left\{  f:(-1,1)\rightarrow\mathbb{C}%
\ \Bigg|\ f\text{ is Lebesgue measurable and}\int_{-1}^{1}\left\vert
f\right\vert ^{2}w_{\alpha,\beta}dx<\infty\right\}  , \label{L^2-Jacobi}%
\end{equation}
with inner product%
\begin{equation}
(f,g)_{\alpha,\beta}:=\int_{-1}^{1}f(x)\overline{g}(x)w_{\alpha,\beta
}(x)dx\quad(f,g\in L_{\alpha,\beta}^{2}(-1,1)). \label{L^2-JacobiIP}%
\end{equation}
In fact, the Jacobi polynomials $\left\{  P_{r}^{(\alpha,\beta)}\right\}
_{r=0}^{\infty}$ form a complete orthogonal sequence in $L_{\alpha,\beta}%
^{2}(-1,1).$ Furthermore, in this space (called the right-definite spectral
setting), there is a self-adjoint operator $A^{(\alpha,\beta)}$ in
$L_{\alpha,\beta}^{2}(-1,1),$ generated by $\ell_{\alpha,\beta}[\cdot],$ that
has the Jacobi polynomials $\{P_{r}^{(\alpha,\beta)}\}_{r=0}^{\infty}$ as
eigenfunctions. Special cases of these polynomials include the Legendre
polynomials $(\alpha=\beta=0),$ the Tchebycheff polynomials of the first kind
$(\alpha=\beta=-1/2)$, the Tchebycheff polynomials of the second kind
$(\alpha=\beta=1/2),$ and the ultraspherical or Gegenbauer polynomials
$(\alpha=\beta).$

The operator $A^{(\alpha,\beta)}$ is an unbounded operator but it is bounded
below by $kI,$ where $I$ denotes the identity operator, in $L_{\alpha,\beta
}^{2}(-1,1).$ Consequently, for $k>0,$ a general left-definite operator theory
developed by Littlejohn and Wellman \cite{LW} applies. In particular, there is
a continuum of Hilbert spaces $\{H_{t}^{(\alpha,\beta)}\}_{t>0}$ where
$H_{t}^{(\alpha,\beta)}$ is called the $t^{th}$ left-definite Hilbert space
associated with the pair $(L_{\alpha,\beta}^{2}(-1,1),A^{(\alpha,\beta)}).$
From the viewpoint of the general theory of orthogonal polynomials, it is
remarkable that the Jacobi polynomials $\{P_{r}^{(\alpha,\beta)}%
\}_{r=0}^{\infty}$ form a complete orthogonal set in $H_{t}^{(\alpha,\beta)}$
for each $t>0.$ The upshot of the left-definite analysis of (\ref{JacobiDE})
is that, for $n\in\mathbb{N},$ the integral composite power $\ell
_{\alpha,\beta}^{n}[\cdot]$ generates the $n^{th}$ left-definite inner product
$(\cdot,\cdot)_{n}^{(\alpha,\beta)}.$ In \cite{EKLWY}, the authors prove the
following result which is the key prerequisite to establishing the
left-definite theory of the\ Jacobi differential expression; it is in this
result where the Jacobi-Stirling numbers are first introduced.

\begin{theorem}
\label{Powers of Jacobi differential expression}Let $n\in\mathbb{N}.$ The
$n^{th}$ composite power of the Jacobi differential expression $($%
\ref{JacobiDE}$),$ in Lagrangian symmetric form, is given by%
\begin{equation}
\ell_{\alpha,\beta}^{n}[y](x)=\dfrac{1}{w_{\alpha,\beta}(x)}\sum_{j=0}%
^{n}(-1)^{j}\left(  c_{j}^{(\gamma)}(n,k)(1-x)^{\alpha+j}(1+x)^{\beta
+j}y^{(j)}(x)\right)  ^{(j)}\quad(x\in(-1,1)), \label{Powers of Jacobi DE}%
\end{equation}
where the coefficients $c_{j}^{(\gamma)}(n,k)$ $(j=0,1,\ldots n)$ are
nonnegative and given by
\begin{equation}
c_{0}^{(\gamma)}(n,k)=\left\{
\begin{array}
[c]{ll}%
0 & \text{if }k=0\\
k^{n} & \text{if }k>0
\end{array}
\right.  \text{ and }c_{j}^{(\gamma)}(n,k):=\left\{
\begin{array}
[c]{ll}%
%TCIMACRO{\QATOPD{\{}{\}}{n}{j}}%
%BeginExpansion
\genfrac{\{}{\}}{0pt}{}{n}{j}%
%EndExpansion
_{\gamma} & \text{if }k=0\\
\sum_{r=0}^{n-j}\binom{n}{r}%
%TCIMACRO{\QATOPD{\{}{\}}{n-r}{j}}%
%BeginExpansion
\genfrac{\{}{\}}{0pt}{}{n-r}{j}%
%EndExpansion
_{\gamma}k^{r} & \text{if }k>0.
\end{array}
\right.  \label{c_j(n,k)}%
\end{equation}
In particular, when $k=0,$
\begin{equation}
\ell_{\alpha,\beta}^{n}[y](x)=\dfrac{1}{w_{\alpha,\beta}(x)}\sum_{j=1}%
^{n}(-1)^{j}\left(
%TCIMACRO{\QATOPD{\{}{\}}{n}{j}}%
%BeginExpansion
\genfrac{\{}{\}}{0pt}{}{n}{j}%
%EndExpansion
_{\gamma}(1-x)^{\alpha+j}(1+x)^{\beta+j}y^{(j)}(x)\right)  ^{(j)}\quad
(x\in(-1,1)). \label{Jacobi powers special case}%
\end{equation}
Furthermore, $%
%TCIMACRO{\QATOPD{\{}{\}}{n}{j}}%
%BeginExpansion
\genfrac{\{}{\}}{0pt}{}{n}{j}%
%EndExpansion
_{\gamma}$ is the coefficient of $x^{n-j}$ in the Maclaurin series expansion
of
\begin{equation}
\prod_{m=1}^{j}\frac{1}{1-m(m+2\gamma-1)x}\quad\left(  \left\vert x\right\vert
<\frac{1}{j(j+2\gamma-1)}\right)  . \label{Rational Generating Function}%
\end{equation}

\end{theorem}

From (\ref{Rational Generating Function}), we see that we can extend the
definition of these numbers to include the initial conditions
\begin{equation}%
%TCIMACRO{\QATOPD{\{}{\}}{n}{0}}%
%BeginExpansion
\genfrac{\{}{\}}{0pt}{}{n}{0}%
%EndExpansion
_{\gamma}=\delta_{n,0}\text{ and }%
%TCIMACRO{\QATOPD{\{}{\}}{0}{j}}%
%BeginExpansion
\genfrac{\{}{\}}{0pt}{}{0}{j}%
%EndExpansion
_{\gamma}=\delta_{j,0}. \label{Initial Conditions}%
\end{equation}

The original motivation in the discovery of the Jacobi-Stirling numbers can be
seen from the columns in Table 1. Indeed, the numbers in the $n^{th}$ column
of Table 1 are precisely the coefficients of the $n^{th}$ power of the Jacobi
differential expression $\ell_{\alpha,\beta}[\cdot]$; for example,
\begin{align*}
\ell_{\alpha,\beta}^{3}  &  [y](x)=\dfrac{1}{w_{\alpha,\beta}(x)}%
[-\mathbf{1}((1-x)^{\alpha+3}(1+x)^{\beta+3})y^{\prime\prime\prime
}(x))^{\prime\prime\prime}\\
&  +((\mathbf{6\gamma+2})(1-x)^{\alpha+2}(1+x)^{\beta+2}y^{\prime\prime
}(x))^{\prime\prime}-(\mathbf{4\gamma}^{2}(1-x)^{\alpha+1}(1+x)^{\beta
+1})y^{\prime}(x))^{\prime}].
\end{align*}
On the other hand, the rows in Table 1 demonstrate the horizontal generating
function (\ref{Rational Generating Function}) for the Jacobi-Stirling numbers.
For example, reading along the row beginning with $j=2,$ we see that%
\[
\prod_{r=1}^{2}\dfrac{1}{1-r(r+2\gamma-1)t}=\mathbf{1}+(\mathbf{6\gamma
+2)}t+(\mathbf{28\gamma}^{2}\mathbf{+20\gamma+4})t^{2}+(\mathbf{120\gamma}%
^{3}\mathbf{+136\gamma}^{2}\mathbf{+56\gamma+8})t^{3}+\ldots.
\]

\section{Jacobi Stirling Numbers Versus Stirling Numbers of the Second
Kind\label{Comparison}}

The Jacobi-Stirling numbers $%
%TCIMACRO{\QATOPD{\{}{\}}{n}{j}}%
%BeginExpansion
\genfrac{\{}{\}}{0pt}{}{n}{j}%
%EndExpansion
_{\gamma}$ are similar in many ways to the classical Stirling numbers of the
second kind, which we denote by $%
%TCIMACRO{\QATOPD{\{}{\}}{n}{j}}%
%BeginExpansion
\genfrac{\{}{\}}{0pt}{}{n}{j}%
%EndExpansion
$. As a first point of comparison, we note that, as reported in \cite{LW} (see
also \cite{ELW-Hermite}), the Stirling numbers of the second kind are the
coefficients of the integral powers of the second-order Laguerre differential
expression $m[\cdot]$, defined by
\[
m[y](x):=\frac{1}{x^{\alpha}e^{-x}}\left(  -x^{\alpha+1}e^{-x}y^{\prime
}(x)\right)  ^{\prime}\quad(x\in(0,\infty)).
\]
Indeed, for each $n\in\mathbb{N},$%
\begin{equation}
m^{n}[y](x)=\frac{1}{x^{\alpha}e^{-x}}\sum_{j=1}^{n}(-1)^{j}\left(
%TCIMACRO{\QATOPD{\{}{\}}{n}{j}}%
%BeginExpansion
\genfrac{\{}{\}}{0pt}{}{n}{j}%
%EndExpansion
x^{\alpha+j}e^{-x}y^{(j)}(x)\right)  ^{(j)}; \label{Laguerre powers}%
\end{equation}
compare (\ref{Jacobi powers special case}) and (\ref{Laguerre powers}).

The following table lists various properties of the Stirling numbers of the
second kind.
\[%
\begin{tabular}
[c]{|l|l|}\hline
Property & Stirling Numbers 2nd Kind\\\hline
Rational GF & ${\displaystyle\prod\limits_{r=1}^{j}\dfrac{1}{1-rx}=\sum
_{n=0}^{\infty}%
%TCIMACRO{\QATOPD{\{}{\}}{n}{j}}%
%BeginExpansion
\genfrac{\{}{\}}{0pt}{}{n}{j}%
%EndExpansion
x^{n-j}}$\\\hline
Vertical RR & ${\displaystyle%
%TCIMACRO{\QATOPD{\{}{\}}{n}{j}}%
%BeginExpansion
\genfrac{\{}{\}}{0pt}{}{n}{j}%
%EndExpansion
=\sum_{r=j}^{n}%
%TCIMACRO{\QATOPD{\{}{\}}{r-1}{j-1}}%
%BeginExpansion
\genfrac{\{}{\}}{0pt}{}{r-1}{j-1}%
%EndExpansion
j^{n-r}}$\\\hline
Triangular RR & $%
%TCIMACRO{\QATOPD{\{}{\}}{n}{j}}%
%BeginExpansion
\genfrac{\{}{\}}{0pt}{}{n}{j}%
%EndExpansion
=%
%TCIMACRO{\QATOPD{\{}{\}}{n-1}{j-1}}%
%BeginExpansion
\genfrac{\{}{\}}{0pt}{}{n-1}{j-1}%
%EndExpansion
+j%
%TCIMACRO{\QATOPD{\{}{\}}{n-1}{j}}%
%BeginExpansion
\genfrac{\{}{\}}{0pt}{}{n-1}{j}%
%EndExpansion
$\\
& $%
%TCIMACRO{\QATOPD{\{}{\}}{n}{0}}%
%BeginExpansion
\genfrac{\{}{\}}{0pt}{}{n}{0}%
%EndExpansion
=\delta_{n,0}$; $%
%TCIMACRO{\QATOPD{\{}{\}}{0}{j}}%
%BeginExpansion
\genfrac{\{}{\}}{0pt}{}{0}{j}%
%EndExpansion
=\delta_{j,0}$\\\hline
Horizontal GF & ${\displaystyle x^{n}=\sum_{j=0}^{n}%
%TCIMACRO{\QATOPD{\{}{\}}{n}{j}}%
%BeginExpansion
\genfrac{\{}{\}}{0pt}{}{n}{j}%
%EndExpansion
(x)_{j}}$, where\\
& $(x)_{j}=x(x-1)\ldots(x-j+1)$\\\hline
Forward Differences & $\Delta^{k}\left(
%TCIMACRO{\QATOPD{\{}{\}}{n}{j}}%
%BeginExpansion
\genfrac{\{}{\}}{0pt}{}{n}{j}%
%EndExpansion
\right)  \geq0$ for $(n\geq j$ and $k\in\mathbb{N}_{0})$\\\hline
\end{tabular}
\ \ \ \ \ \ \ \ \ \
\]

\begin{center}
\textbf{Table 2: Properties of the classical Stirling numbers of the second
kind\medskip}
\end{center}

Regarding the `forward difference' entry in the above Table, recall that the
forward difference of a sequence of numbers $\{x_{n}\}_{n=0}^{\infty}$ is the
sequence $\{\Delta x_{n}\}_{n=0}^{\infty}$ defined by%
\[
\Delta x_{n}:=x_{n+1}-x_{n}\quad(n\in\mathbb{N}_{0}).
\]
Higher-order forward differences are defined recursively by%
\[
\Delta^{k}x_{n}:=\Delta(\Delta^{k-1}x_{n})=\sum_{m=0}^{k}\binom{k}{m}%
(-1)^{m}x_{n+k-m}.
\]
Comtet \cite[Propostion, p. 749]{Comtet-1972} indicates the forward
differences inequalities for the Stirling numbers. Two excellent sources for
the Stirling numbers of the second kind, and their properties, are the books
by Abramowitz and Stegun \cite{AS} and Comtet \cite[Chapter V]{Comtet}.

We now state a theorem that has the comparable properties of the
Jacobi-Stirling numbers; the reader will immediately observe the close
similarities between Stirling numbers of the second kind and Jacobi-Stirling
numbers. The details and proofs of most of these properties are given in
\cite{Andrews-Gawronski-Littlejohn} in the case of the Legendre-Stirling
numbers $(\alpha=\beta=0).$ Since the proofs are almost identical, we will
only sketch proofs when necessary.

\begin{theorem}
\label{Properties of the Jacobi-Stirling numbers} The Jacobi-Stirling numbers
$%
%TCIMACRO{\QATOPD{\{}{\}}{n}{j}}%
%BeginExpansion
\genfrac{\{}{\}}{0pt}{}{n}{j}%
%EndExpansion
_{\gamma}$ have the following properties.

\begin{enumerate}
\item[(i)] $($Rational Generating Function$)$ For all $j\in\mathbb{N}_{0}$,
\[
\prod\limits_{r=1}^{j}\dfrac{1}{1-r(r+2\gamma-1)x}=\sum_{n=0}^{\infty}%
%TCIMACRO{\QATOPD{\{}{\}}{n}{j}}%
%BeginExpansion
\genfrac{\{}{\}}{0pt}{}{n}{j}%
%EndExpansion
_{\gamma}x^{n-j}=\sum_{n=j}^{\infty}%
%TCIMACRO{\QATOPD{\{}{\}}{n}{j}}%
%BeginExpansion
\genfrac{\{}{\}}{0pt}{}{n}{j}%
%EndExpansion
_{\gamma}x^{n-j}\hspace{30pt}\left(  \left\vert x\right\vert <\dfrac
{1}{j(j+2\gamma-1)}\right)  ;
\]
in particular, for each $n,j\in\mathbb{N},$ $%
%TCIMACRO{\QATOPD{\{}{\}}{n}{j}}%
%BeginExpansion
\genfrac{\{}{\}}{0pt}{}{n}{j}%
%EndExpansion
_{\gamma}$ is a polynomial in $\gamma$ with nonnegative integer coefficients.

\item[(ii)] $($Vertical Recurrence Relation$)$ For all $n,j\in\mathbb{N}_{0}%
$,
\[%
%TCIMACRO{\QATOPD{\{}{\}}{n}{j}}%
%BeginExpansion
\genfrac{\{}{\}}{0pt}{}{n}{j}%
%EndExpansion
_{\gamma}=\sum_{r=j}^{n}%
%TCIMACRO{\QATOPD{\{}{\}}{r-1}{j-1}}%
%BeginExpansion
\genfrac{\{}{\}}{0pt}{}{r-1}{j-1}%
%EndExpansion
_{\gamma}(j(j+2\gamma-1))^{n-r}.
\]

\item[(iii)] $($Triangular Recurrence Relation$)$ For all $n,j\in\mathbb{N}$,
\[%
%TCIMACRO{\QATOPD{\{}{\}}{n}{j}}%
%BeginExpansion
\genfrac{\{}{\}}{0pt}{}{n}{j}%
%EndExpansion
_{\gamma}=%
%TCIMACRO{\QATOPD{\{}{\}}{n-1}{j-1}}%
%BeginExpansion
\genfrac{\{}{\}}{0pt}{}{n-1}{j-1}%
%EndExpansion
_{\gamma}+j(j+2\gamma-1)%
%TCIMACRO{\QATOPD{\{}{\}}{n-1}{j}}%
%BeginExpansion
\genfrac{\{}{\}}{0pt}{}{n-1}{j}%
%EndExpansion
_{\gamma},
\]
and for all $n,j\in\mathbb{N}_{0}$,
\[%
%TCIMACRO{\QATOPD{\{}{\}}{n}{0}}%
%BeginExpansion
\genfrac{\{}{\}}{0pt}{}{n}{0}%
%EndExpansion
_{\gamma}=\delta_{n,0}\hspace{40pt}%
%TCIMACRO{\QATOPD{\{}{\}}{0}{j}}%
%BeginExpansion
\genfrac{\{}{\}}{0pt}{}{0}{j}%
%EndExpansion
_{\gamma}=\delta_{j,0}.
\]

\item[(iv)] $($Horizontal Generating Function$)$ For all $n\in\mathbb{N}_{0}%
$,
\[
x^{n}=\sum_{j=0}^{n}%
%TCIMACRO{\QATOPD{\{}{\}}{n}{j}}%
%BeginExpansion
\genfrac{\{}{\}}{0pt}{}{n}{j}%
%EndExpansion
_{\gamma}\langle x\rangle_{j}^{(\gamma)},
\]
where $\langle x\rangle_{j}^{(\gamma)}$ is the generalized falling factorial
defined by
\[
\langle x\rangle_{j}^{(\gamma)}:=\text{ }\left\{
\begin{array}
[c]{cl}%
1 & \text{if }j=0\\
\prod_{m=0}^{j-1}(x-m(m+2\gamma-1)) & \text{if }j\in\mathbb{N}.
\end{array}
\right.
\]

\item[(v)] $($Forward Differences$)$ For all $k\in\mathbb{N}_{0},$
\[
\Delta^{k}\left(  \dfrac{%
%TCIMACRO{\QATOPD{\{}{\}}{n}{j}}%
%BeginExpansion
\genfrac{\{}{\}}{0pt}{}{n}{j}%
%EndExpansion
_{\gamma}}{(2\gamma)^{n}}\right)  \geq0\quad(n\geq j).
\]

\end{enumerate}
\end{theorem}

\begin{proof}
The rational generating function, given in $(i)$, is discussed in Theorem
\ref{Powers of Jacobi differential expression}; the complete proof of $(i)$ is
given in \cite[Theorem 4.1]{EKLWY}; we also refer to \cite[Section
4.2]{Gelineau-Zeng} where the authors determine this rational generating
function in a different and simple way. The vertical recurrence relation in
$(ii)$ follows from $(i)$; in fact, the proof is similar to that given in the
Legendre-Stirling case which can be found in \cite[Theorem 5.2]%
{Andrews-Gawronski-Littlejohn}. Since the triangular recurrence relation is
important in the combinatorial interpretation of the Jacobi-Stirling numbers,
which we discuss in the next section, we give a proof of $(iii).$ The initial
conditions given in $(iii)$ are part of the definition of $%
%TCIMACRO{\QATOPD{\{}{\}}{n}{j}}%
%BeginExpansion
\genfrac{\{}{\}}{0pt}{}{n}{j}%
%EndExpansion
_{\gamma},$ given in $($\ref{Initial Conditions}$).$ From $(i),$ we see that%
\[
\sum_{n=j-1}^{\infty}%
%TCIMACRO{\QATOPD{\{}{\}}{n}{j-1}}%
%BeginExpansion
\genfrac{\{}{\}}{0pt}{}{n}{j-1}%
%EndExpansion
_{\gamma}x^{n-j+1}=(1-j(j+2\gamma-1)x)\sum_{n=j}^{\infty}%
%TCIMACRO{\QATOPD{\{}{\}}{n}{j}}%
%BeginExpansion
\genfrac{\{}{\}}{0pt}{}{n}{j}%
%EndExpansion
_{\gamma}x^{n-j};
\]
shifting the index on the sum on the left-hand side, and carrying out the
multiplication on the right-side yields%
\begin{align*}
\sum_{n=j}^{\infty}%
%TCIMACRO{\QATOPD{\{}{\}}{n-1}{j-1}}%
%BeginExpansion
\genfrac{\{}{\}}{0pt}{}{n-1}{j-1}%
%EndExpansion
_{\gamma}x^{n-j}  &  =\sum_{n=j}^{\infty}%
%TCIMACRO{\QATOPD{\{}{\}}{n}{j}}%
%BeginExpansion
\genfrac{\{}{\}}{0pt}{}{n}{j}%
%EndExpansion
_{\gamma}x^{n-j}-j(j+2\gamma-1)\sum_{n=j}^{\infty}%
%TCIMACRO{\QATOPD{\{}{\}}{n}{j}}%
%BeginExpansion
\genfrac{\{}{\}}{0pt}{}{n}{j}%
%EndExpansion
_{\gamma}x^{n-j+1}\\
&  =\sum_{n=j}^{\infty}%
%TCIMACRO{\QATOPD{\{}{\}}{n}{j}}%
%BeginExpansion
\genfrac{\{}{\}}{0pt}{}{n}{j}%
%EndExpansion
_{\gamma}x^{n-j}-j(j+2\gamma-1)\sum_{n=j}^{\infty}%
%TCIMACRO{\QATOPD{\{}{\}}{n-1}{j}}%
%BeginExpansion
\genfrac{\{}{\}}{0pt}{}{n-1}{j}%
%EndExpansion
_{\gamma}x^{n-j}%
\end{align*}
and this implies $(iii)$. The proof of $(iv)$ is similar to the proof of the
horizontal generating function for the Legendre-Stirling numbers given in
\cite[Theorem 5.4]{Andrews-Gawronski-Littlejohn}. The proof of $(v)$ is very
similar to the proof given in \cite[Theorem 5.1]{Andrews-Gawronski-Littlejohn}
in the case of the Legendre-Stirling numbers.
\end{proof}

For a different approach to parts (i), (iii), and (iv) see also
\cite{Gelineau-Zeng}.

From Table 2 and Theorem \ref{Properties of the Jacobi-Stirling numbers},
observe that the rational generating functions for the Stirling numbers of the
second kind (which we again note are connected to powers of the Laguerre
differential expression in (\ref{Laguerre powers})) and the Jacobi-Stirling
numbers (associated with the powers of the Jacobi differential expression in
(\ref{Jacobi powers special case})) involve, respectively, the coefficients
$r$ and $r(r+\alpha+\beta+1)$ in the denominators of these products.
Remarkably, and somewhat mysteriously, these coefficients are, respectively,
the eigenvalues that produce the Laguerre and Jacobi polynomial solutions of
degree $r$ in, respectively, the Laguerre and Jacobi differential equations.
Computing the integral composite powers of both the Laguerre and Jacobi
differential equations is entirely \textit{algebraic}, and one would not
initially expect these calculations to involve spectral theory. Furthermore,
each self-adjoint operator in $L_{\alpha,\beta}^{2}(-1,1)$ generated by the
Jacobi differential expression $\ell_{\alpha,\beta}[\cdot]$ has a discrete
(eigenvalues) spectrum only. It is natural to ask why the horizontal
generating function specifically involves the eigenvalues $\{r(r+\alpha
+\beta+1)\}_{r=0}^{\infty}$ associated with the operator $A^{(\alpha,\beta)}$
(from Section \ref{Background}) that has eigenfunctions $\{P_{r}%
^{(\alpha,\beta)}\}_{r=0}^{\infty}$ and not the set of eigenvalues of another
self-adjoint operator. It seems that there is an interesting connection here
that deserves further attention.

The last point that we wish to make in this section is a very interesting
connection between Legendre-Stirling numbers $%
%TCIMACRO{\QATOPD{\{}{\}}{n}{j}}%
%BeginExpansion
\genfrac{\{}{\}}{0pt}{}{n}{j}%
%EndExpansion
_{1}$ and the classical Stirling numbers $S_{n}^{(j)}$ of the second kind. In
\cite[equation (3.1)]{Andrews-Gawronski-Littlejohn}, the authors prove that%
\[%
%TCIMACRO{\QATOPD{\{}{\}}{n}{j}}%
%BeginExpansion
\genfrac{\{}{\}}{0pt}{}{n}{j}%
%EndExpansion
_{1}=\dfrac{1}{(2j)!}\sum_{m=0}^{2j}(-1)^{m}\binom{2j}{m}\left(
(j-m)(j+1-m)\right)  ^{n}.
\]
It follows then that
\begin{align*}%
%TCIMACRO{\QATOPD{\{}{\}}{n}{j}}%
%BeginExpansion
\genfrac{\{}{\}}{0pt}{}{n}{j}%
%EndExpansion
_{1}  &  =\dfrac{\partial^{2n}}{(\partial x)^{n}(\partial y)^{n}}\dfrac
{1}{(2j)!}\sum_{m=0}^{2j}(-1)^{m}\binom{2j}{m}e^{(j-m)x+(j+1-m)y}\left\vert
_{x=y=0}\right. \\
&  =\dfrac{1}{(2j)!}\dfrac{\partial^{2n}}{(\partial x)^{n}(\partial y)^{n}%
}e^{jx+(j+1)y}(1-e^{-(x+y)})^{2j}\left\vert _{x=y=0}\right. \\
&  =\dfrac{\partial^{2n}}{(\partial x)^{n}(\partial y)^{n}}e^{-jx}%
e^{-(j-1)y}\dfrac{(e^{x+y}-1)^{2j}}{(2j)!}\left\vert _{x=y=0}\right. \\
&  =\dfrac{\partial^{2n}}{(\partial x)^{n}(\partial y)^{n}}e^{-jx}%
e^{-(j-1)y}\phi_{2j}(x+y)\left\vert _{x=y=0}\right.  ,
\end{align*}
where $\phi_{j}(\cdot)$ is the vertical generating function for the Stirling
numbers of the second kind; that is,%
\[
\phi_{j}(t):=\sum_{n=0}^{\infty}\dfrac{S_{n}^{(j)}}{n!}t^{n}=\dfrac
{(\exp(t)-1)^{j}}{j!}.
\]
Continuing, we find that
\begin{align*}%
%TCIMACRO{\QATOPD{\{}{\}}{n}{j}}%
%BeginExpansion
\genfrac{\{}{\}}{0pt}{}{n}{j}%
%EndExpansion
_{1}  &  =\dfrac{\partial^{2n}}{(\partial x)^{n}(\partial y)^{n}}%
e^{-jx}e^{-(j-1)y}\phi_{2j}(x+y)\left\vert _{x=y=0}\right. \\
&  =\left(  \frac{\partial}{\partial x}\right)  ^{n}e^{-jx}\sum_{\nu=0}%
^{n}\binom{n}{\nu}\left\{  \left(  \frac{\partial}{\partial y}\right)
^{n-\nu}e^{-(j-1)y}\right\}  \left\{  \left(  \frac{\partial}{\partial
y}\right)  ^{\nu}\phi_{2j}(x+y)\right\}  \left\vert _{x=y=0}\right. \\
&  =\sum_{\nu,\mu=0}^{n}\binom{n}{\nu}\binom{n}{\mu}\left(  -(j-1)\right)
^{n-\nu}(-j)^{n-\mu}\phi_{2j}^{(\nu+\mu)}(0)\\
&  =\sum_{\nu,\mu=0}^{n}\binom{n}{\nu}\binom{n}{\mu}\left(  -(j-1)\right)
^{n-\nu}(-j)^{n-\mu}S_{\nu+\mu}^{(2j)}.
\end{align*}
Using the forward difference operator, this latter formula for the
Legendre-Stirling numbers may be written in the compact form%
\begin{align}%
%TCIMACRO{\QATOPD{\{}{\}}{n}{1}}%
%BeginExpansion
\genfrac{\{}{\}}{0pt}{}{n}{1}%
%EndExpansion
_{1}  &  =\Delta^{n}S_{n}^{(2)}=2^{n-1}\label{LS-S1}\\%
%TCIMACRO{\QATOPD{\{}{\}}{n}{j}}%
%BeginExpansion
\genfrac{\{}{\}}{0pt}{}{n}{j}%
%EndExpansion
_{1}  &  =(j-1)^{n}j^{n}\Delta_{x}^{n}\Delta_{y}^{n}\frac{S_{x+y}^{(2j)}%
}{(j-1)^{x}j^{y}}\left\vert _{x=y=0}\right.  \text{ when }j\geq2,
\label{LS-S2}%
\end{align}
with the obvious meaning of $\Delta_{x}$ and $\Delta_{y}.$ It would be
interesting to see if there is a similar connection between the Stirling
numbers of the second kind and the Jacobi-Stirling numbers.

\section{A Combinatorial Interpretation of the Jacobi-Stirling
Numbers\label{Main Results1}}

The Stirling number of the second kind $%
%TCIMACRO{\QATOPD{\{}{\}}{n}{j}}%
%BeginExpansion
\genfrac{\{}{\}}{0pt}{}{n}{j}%
%EndExpansion
$ is the number of set partitions of $\{1,2,\ldots,n\}$ into $j$ nonempty
blocks. That is, $%
%TCIMACRO{\QATOPD{\{}{\}}{n}{j}}%
%BeginExpansion
\genfrac{\{}{\}}{0pt}{}{n}{j}%
%EndExpansion
$ is the number of ways of placing $n$ objects into $j$ non-empty,
indistinguishable sets; for a full account of their properties, see \cite{AS}
and Comtet \cite[Chapter V]{Comtet}. With this in mind, it is natural to ask
for a similar combinatorial interpretation of the Jacobi-Stirling number $%
%TCIMACRO{\QATOPD{\{}{\}}{n}{j}}%
%BeginExpansion
\genfrac{\{}{\}}{0pt}{}{n}{j}%
%EndExpansion
_{\gamma}$. Indeed, Andrews and Littlejohn \cite{Andrews-Littlejohn} have
generalized the notion of a set partition to give a combinatorial
interpretation of the Legendre-Stirling number $%
%TCIMACRO{\QATOPD{\{}{\}}{n}{j}}%
%BeginExpansion
\genfrac{\{}{\}}{0pt}{}{n}{j}%
%EndExpansion
_{1}$; one might call Andrews and Littlejohn's generalized set partitions
\emph{Legendre-Stirling set partitions}. More recently, Gelineau and Zeng
\cite{Gelineau-Zeng} have found a statistic on Legendre-Stirling set
partitions which allows them to interpret $%
%TCIMACRO{\QATOPD{\{}{\}}{n}{j}}%
%BeginExpansion
\genfrac{\{}{\}}{0pt}{}{n}{j}%
%EndExpansion
_{\gamma}$ as a generating function over Legendre-Stirling set partitions. In
this section we generalize Legendre-Stirling set partitions still further, to
obtain objects we will call \emph{Jacobi-Stirling set partitions}. When
$\gamma=1$ it will be clear that the Jacobi-Stirling set partitions are in
fact Legendre-Stirling set partitions, and we will show that for any positive
integer $\gamma$, the Jacobi-Stirling numbers count Jacobi-Stirling set partitions.

To describe a combinatorial interpretation of the Jacobi-Stirling number $%
%TCIMACRO{\QATOPD{\{}{\}}{n}{j}}%
%BeginExpansion
\genfrac{\{}{\}}{0pt}{}{n}{j}%
%EndExpansion
_{\gamma}$, let $[n]_{2}$ denote the set $\{1_{1},1_{2},2_{1},2_{2}%
,\ldots,n_{1},n_{2}\}$, which contains two copies of each positive integer
between $1$ and $n$; we may say that these are the integers $\{1,2,\ldots,n\}$
with two different colors. By convention $[0]_{2}$ is the empty set.

\begin{definition}
For all $n,j,\gamma\in\mathbb{N}_{0}$, a \emph{Jacobi-Stirling set partition}
of $[n]_{2}$ into $\gamma$ zero blocks and $j$ nonzero blocks is an ordinary
set partition of $[n]_{2}$ into $j+\gamma$ blocks for which the following
conditions hold:

\begin{enumerate}
\item $\gamma$ of our blocks, called the \emph{zero blocks}, are
distinguishable, but all other blocks are indistinguishable.

\item The zero blocks may be empty, but all other blocks are nonempty.

\item The union of the zero blocks may not contain both copies of any number.

\item Each nonzero block contains both copies of the smallest number it
contains, but does not contain both copies of any other number.
\end{enumerate}
\end{definition}

\begin{example}
\label{Example 1} As we see in the table below, there are 20 Jacobi-Stirling
set partitions of $[3]_{2}$ into $\gamma= 3$ zero blocks and $j = 2$ nonzero
blocks.
\[%
\begin{tabular}
[c]{|l|l|l|l|}\hline
Zero Boxes & Nonzero Boxes & Zero Boxes & Nonzero Boxes\\\hline
$\varnothing,\varnothing,\varnothing$ & $\{1_{1},1_{2},3_{1}\},\{2_{1}%
,2_{2},3_{2}\}$ & $\{3_{1}\},\varnothing,\varnothing$ & $\{1_{1},1_{2}%
,3_{2}\},\{2_{1},2_{2}\}$\\\hline
$\varnothing,\varnothing,\varnothing$ & $\{1_{1},1_{2},3_{1}\},\{2_{1}%
,2_{2},3_{2}\}$ & $\varnothing,\varnothing,\{3_{2}\}$ & $\{1_{1},1_{2}%
,3_{1}\},\{2_{1},2_{2}\}$\\\hline
$\varnothing,\varnothing,\{3_{1}\}$ & $\{1_{1},1_{2}\},\{2_{1},2_{2},3_{2}\}$
& $\varnothing,\{3_{2}\},\varnothing$ & $\{1_{1},1_{2},3_{1}\},\{2_{1}%
,2_{2}\}$\\\hline
$\varnothing,\{3_{1}\},\varnothing$ & $\{1_{1},1_{2}\},\{2_{1},2_{2},3_{2}\}$
& $\{3_{2}\},\varnothing,\varnothing$ & $\{1_{1},1_{2},3_{1}\},\{2_{1}%
,2_{2}\}$\\\hline
$\{3_{1}\},\varnothing,\varnothing$ & $\{1_{1},1_{2}\},\{2_{1},2_{2},3_{2}\}$
& $\{2_{1}\},\varnothing,\varnothing$ & $\{1_{1},1_{2},2_{2}\},\{3_{1}%
,3_{2}\}$\\\hline
$\varnothing,\varnothing,\{3_{2}\}$ & $\{1_{1},1_{2}\},\{2_{1},2_{2},3_{1}\}$
& $\varnothing,\{2_{1}\},\varnothing$ & $\{1_{1},1_{2},2_{2}\},\{3_{1}%
,3_{2}\}$\\\hline
$\varnothing,\{3_{2}\},\varnothing$ & $\{1_{1},1_{2}\},\{2_{1},2_{2},3_{1}\}$
& $\varnothing,\varnothing,\{2_{1}\}$ & $\{1_{1},1_{2},2_{2}\},\{3_{1}%
,3_{2}\}$\\\hline
$\{3_{2}\},\varnothing,\varnothing$ & $\{1_{1},1_{2}\},\{2_{1},2_{2},3_{1}\}$
& $\{2_{2}\},\varnothing,\varnothing$ & $\{1_{1},1_{2},2_{1}\},\{3_{1}%
,3_{2}\}$\\\hline
$\varnothing,\varnothing,\{3_{1}\}$ & $\{1_{1},1_{2},3_{2}\},\{2_{1},2_{2}\}$
& $\varnothing,\{2_{2}\},\varnothing$ & $\{1_{1},1_{2},2_{1}\},\{3_{1}%
,3_{2}\}$\\\hline
$\varnothing,\{3_{1}\},\varnothing$ & $\{1_{1},1_{2},3_{2}\},\{2_{1},2_{2}\}$
& $\varnothing,\varnothing,\{2_{2}\}$ & $\{1_{1},1_{2},2_{1}\},\{3_{1}%
,3_{2}\}$\\\hline
\end{tabular}
\ \ \ \ \ \ \ \
\]

\end{example}

As we show next, Jacobi-Stirling numbers count Jacobi-Stirling set partitions.

\begin{theorem}
\label{Combinatorial Interpretation} For all $n,j,\gamma\in\mathbb{N}_{0}$,
the Jacobi-Stirling number $%
%TCIMACRO{\QATOPD{\{}{\}}{n}{j}}%
%BeginExpansion
\genfrac{\{}{\}}{0pt}{}{n}{j}%
%EndExpansion
_{\gamma}$ is the number of Jacobi-Stirling set partitions of $[n]_{2}$ into
$\gamma$ zero blocks and $j$ nonzero blocks.
\end{theorem}

\begin{proof}
For all $n,j,\gamma\in\mathbb{N}_{0}$, let $P(n,j,\gamma)$ denote the number
of Jacobi-Stirling set partitions of $[n]_{2}$ into $\gamma$ zero blocks and
$j$ nonzero blocks. Since the Jacobi-Stirling numbers are determined by the
initial conditions and recurrence relation in Theorem
\ref{Properties of the Jacobi-Stirling numbers}(iii), it is sufficient to show
that $P(n,j,\gamma)$ satisfies the same initial conditions and recurrence relation.

Since the union of the zero blocks of a Jacobi-Stirling set partition cannot
contain both $1_{1}$ and $1_{2}$, we see that $P(n,0,\gamma) = 0$ if $n > 0$.
On the other hand, for each $\gamma$ there is one Jacobi-Stirling set
partition of the empty set, so $P(0,0,\gamma) = 1$. Finally, since the nonzero
blocks of a Jacobi-Stirling set partition must be nonempty, we see that
$P(0,j,\gamma) = 0$ if $j > 0$. Therefore $P(n,j,\gamma)$ satisfies the same
initial conditions as $%
%TCIMACRO{\QATOPD{\{}{\}}{n}{j}}%
%BeginExpansion
\genfrac{\{}{\}}{0pt}{}{n}{j}%
%EndExpansion
_{\gamma}$.

To see that $P(n,j,\gamma)$ satisfies the same recurrence relation, we note
that for $n \ge1$, Jacobi-Stirling set partitions of $[n]_{2}$ into $\gamma$
zero blocks and $j$ nonzero blocks come in two disjoint types:

\begin{enumerate}
\item[(i)] those in which $n_{1}$ and $n_{2}$ are in the same box;

\item[(ii)] those in which $n_{1}$ and $n_{2}$ are in different boxes.
\end{enumerate}

Each Jacobi-Stirling set partition in class (i) can be uniquely constructed
from a Jacobi-Stirling set partition of $[n-1]_{2}$ into $\gamma$ zero blocks
and $j-1$ nonzero blocks by appending a nonzero block containing only $n_{1}$
and $n_{2}$. Therefore there are $P(n-1,j-1,\gamma)$ partitions in class (i).
On the other hand, each Jacobi-Stirling set partition in class (ii) can be
uniquely constructed from a Jacobi-Stirling set partition of $[n-1]_{2}$ into
$\gamma$ zero blocks and $j-1$ nonzero blocks by either inserting $n_{1}$ into
a zero box and inserting $n_{2}$ into a nonzero box, or by inserting $n_{1}$
into an nonzero box and inserting $n_{2}$ into any box not containing $n_{1}$.
There are $\gamma j P(n-1,j,\gamma)$ partitions of the first type, and $j
(\gamma+j-1) P(n-1,j,\gamma)$ partitions of the second type, so there are $j
(j+2\gamma-1) P(n-1,j,\gamma)$ partitions in class (ii). Therefore
$P(n,j,\gamma) = P(n-1,j-1,\gamma) + j (j+2\gamma-1) P(n-1,j,\gamma)$, so
$P(n,j,\gamma)$ satisfies the same recurrence relation as $%
%TCIMACRO{\QATOPD{\{}{\}}{n}{j}}%
%BeginExpansion
\genfrac{\{}{\}}{0pt}{}{n}{j}%
%EndExpansion
_{\gamma}$, and the result follows.
\end{proof}

\begin{example}
\label{Example 2}In \cite[Example 4.4]{Andrews-Littlejohn}, the authors showed
that $%
%TCIMACRO{\QATOPD{\{}{\}}{n}{1}}%
%BeginExpansion
\genfrac{\{}{\}}{0pt}{}{n}{1}%
%EndExpansion
_{1}=2^{n-1}.$ This argument generalizes to show that $%
%TCIMACRO{\QATOPD{\{}{\}}{n}{1}}%
%BeginExpansion
\genfrac{\{}{\}}{0pt}{}{n}{1}%
%EndExpansion
_{\gamma}=(2\gamma)^{n-1}$ for all $n\in\mathbb{N}.$
\end{example}

\begin{example}
\label{Example 3}In this example we give a direct combinatorial proof that $%
%TCIMACRO{\QATOPD{\{}{\}}{n}{n-1}}%
%BeginExpansion
\genfrac{\{}{\}}{0pt}{}{n}{n-1}%
%EndExpansion
_{\gamma}=2\binom{n}{3}+2\gamma\binom{n}{2},$ where $\binom{n}{j}$ denotes the
usual binomial coefficient. By Theorem \ref{Combinatorial Interpretation}, the
quantity $%
%TCIMACRO{\QATOPD{\{}{\}}{n}{n-1}}%
%BeginExpansion
\genfrac{\{}{\}}{0pt}{}{n}{n-1}%
%EndExpansion
_{\gamma}$ is the number of Jacobi-Stirling partitions of $[n]_{2}$ into
$\gamma$ zero blocks and $n-1$ nonzero blocks. In such a partition there must
be exactly one number $k$ which is not the minimum in its block, and at least
one copy of $k$ is in a nonzero block. If both copies of $k$ are in nonzero
blocks, then we may construct our partition uniquely by choosing the elements
$i<j<k$ of these blocks, and then placing $k_{1}$ and $k_{2}$ in blocks with
minima $i$ and $j$. Thus there are $2\binom{n}{3}$ of these partitions.
Alternatively, if one copy of $k$ is in a zero block, then we may construct
our partition uniquely by choosing the elements $i<k$ of the nonzero block
containing $k$, constructing that block with one copy of $k$, and placing the
other copy of $k$ in one of the $\gamma$ zero blocks. Therefore there are
$2\gamma\binom{n}{2}$ of these partitions. Combining these two counts, we find
that $%
%TCIMACRO{\QATOPD{\{}{\}}{n}{n-1}}%
%BeginExpansion
\genfrac{\{}{\}}{0pt}{}{n}{n-1}%
%EndExpansion
_{\gamma}=2\binom{n}{3}+2\gamma\binom{n}{2}$, as claimed.
\end{example}

In a development independent of this work, Mongelli has recently given another
combinatorial interpretation of the Jacobi-Stirling numbers \cite{Mongelli2}.
In fact, Mongelli gives a combinatorial interpretation of a general family of
numbers, which includes the Jacobi-Stirling numbers as a special case.
Translated into our setting, Mongelli shows inductively that if $z =
2\gamma-1$ is a positive integer then $%
%TCIMACRO{\QATOPD{\{}{\}}{n}{j}}%
%BeginExpansion
\genfrac{\{}{\}}{0pt}{}{n}{j}%
%EndExpansion
_{\gamma}$ is the number of set partitions of $[n]_{2}$ into $z$ zero blocks
and $j$ nonzero blocks for which the following conditions hold.

\begin{enumerate}
\item[(1)] The zero blocks are distinguishable, but the nonzero blocks are indistinguishable.

\item[(2)] The zero blocks may be empty, but the nonzero blocks are nonempty.

\item[(3)] No zero block may contain the first copy of any number.

\item[(4)] Each nonzero block contains both copies of the smallest number it contains.
\end{enumerate}

For convenience, we call one of Mongelli's set partitions a \emph{long
Jacobi-Stirling set partition}.

For any $n,j,\gamma\in{\mathbb{N}}_{0}$, we can give a bijection between the
associated Jacobi-Stirling set partitions and the associated long
Jacobi-Stirling set partitions. To do this, suppose a given long
Jacobi-Stirling set partition has zero blocks $Z_{0}, Z_{1}, \ldots,
Z_{2\gamma-2}$ and nonzero blocks $B_{1},\ldots, B_{j}$. To obtain a
Jacobi-Stirling set partition, for each $i$, $1 \le i \le n$, we do the following.

\begin{itemize}
\item If $i_{1}$ and $i_{2}$ are in different nonzero blocks, then we leave
$i_{1}$ and $i_{2}$ where they are.

\item If $i_{1}$ and $i_{2}$ are in the same nonzero block and they are the
smallest element of that block, then we leave $i_{1}$ and $i_{2}$ where they are.

\item If $i_{1}$ and $i_{2}$ are in the same nonzero block $B_{k}$ and they
are not the smallest element of that block, then we put $i_{1}$ in $Z_{0}$ and
we leave $i_{2}$ where it is.

\item If $i_{2} \in Z_{k}$ for $0 \le k \le\gamma-1$ then we leave $i_{1}$ and
$i_{2}$ where they are.

\item If $i_{2} \in Z_{\gamma+k-1}$ for $1 \le k \le\gamma-1$ and $i_{1} \in
B_{m}$ then we put $i_{2}$ in $B_{m}$ and we put $i_{1}$ in $Z_{k}$.
\end{itemize}

Now $Z_{\gamma}, Z_{\gamma+1},\ldots,Z_{2\gamma-2}$ are empty, so we remove
them from our set partition. The result is a Jacobi-Stirling set partition,
and we can reverse the process step-by-step to recover the original long
Jacobi-Stirling set partition.

\section{The Jacobi-Stirling Numbers of the First and Second Kinds\label{JS1}}

We saw in Theorems \ref{Properties of the Jacobi-Stirling numbers} and
\ref{Combinatorial Interpretation} that the Jacobi-Stirling numbers are
natural analogues of the Stirling numbers of the second kind (see also
\cite{Gelineau-Zeng} and \cite{Mongelli2}). For instance, both of these
families of numbers have elegant horizontal generating functions: the Stirling
numbers of the second kind satisfy
\begin{equation}
x^{n}=\sum_{j=0}^{n}%
%TCIMACRO{\QATOPD{\{}{\}}{n}{j}}%
%BeginExpansion
\genfrac{\{}{\}}{0pt}{}{n}{j}%
%EndExpansion
(x)_{j},\hspace{30pt}(n\in\mathbb{N}_{0}), \label{eqn:S2expansion}%
\end{equation}
where $(x)_{j}$ is the falling factorial defined by
\begin{equation}
(x)_{j}:=%
\begin{cases}
x(x-1)\cdots(x-j+1) & \text{if }j\in\mathbb{N}\\
1 & \text{if }j=0,
\end{cases}
\label{eqn:fallingfact}%
\end{equation}
while the Jacobi-Stirling numbers satisfy
\begin{equation}
x^{n}=\sum_{j=0}^{n}%
%TCIMACRO{\QATOPD{\{}{\}}{n}{j}}%
%BeginExpansion
\genfrac{\{}{\}}{0pt}{}{n}{j}%
%EndExpansion
_{\gamma}\langle x\rangle_{j}^{(\gamma)}, \label{eqn:js2expansion}%
\end{equation}
where $\langle x\rangle_{j}^{(\gamma)}$ is the generalized falling factorial
defined by
\begin{equation}
\langle x\rangle_{j}^{(\gamma)}:=%
\begin{cases}
1 & \text{if }j=0\\
\prod_{m=0}^{j-1}(x-m(m+2\gamma-1)) & \text{if }j\in\mathbb{N}.
\end{cases}
\label{eqn:jsfalling}%
\end{equation}
In the case of the Stirling numbers we can invert (\ref{eqn:S2expansion}) to
obtain the Stirling numbers of the first kind. In particular, these (signless)
Stirling numbers of the first kind $%
%TCIMACRO{\QATOPD{[}{]}{n}{j}}%
%BeginExpansion
\genfrac{[}{]}{0pt}{}{n}{j}%
%EndExpansion
$ may be defined by
\begin{equation}
(x)_{n}=\sum_{j=0}^{n}(-1)^{n+j}%
%TCIMACRO{\QATOPD{[}{]}{n}{j}}%
%BeginExpansion
\genfrac{[}{]}{0pt}{}{n}{j}%
%EndExpansion
x^{j},\hspace{40pt}(n\in\mathbb{N}_{0}), \label{eqn:S1expansion}%
\end{equation}
where $(x)_{n}$ is the falling factorial given in (\ref{eqn:fallingfact}).
Similarly, if we invert (\ref{eqn:js2expansion}) then we obtain a collection
of numbers we will call the (\emph{signless})\emph{ Jacobi-Stirling numbers of
the first kind}, and which we will denote by $%
%TCIMACRO{\QATOPD{[}{]}{n}{j}}%
%BeginExpansion
\genfrac{[}{]}{0pt}{}{n}{j}%
%EndExpansion
_{\gamma}$. Specifically,
\begin{equation}
\langle x\rangle_{n}^{(\gamma)}=\sum_{j=0}^{n}(-1)^{n+j}%
%TCIMACRO{\QATOPD{[}{]}{n}{j}}%
%BeginExpansion
\genfrac{[}{]}{0pt}{}{n}{j}%
%EndExpansion
_{\gamma}x^{j}, \label{eqn:JS1 poly}%
\end{equation}
where $\langle x\rangle_{j}^{(\gamma)}$ is the generalized falling factorial
defined in (\ref{eqn:jsfalling}). We immediately obtain the following
biorthogonality relationships:
\[
\sum_{m\leq j\leq n}(-1)^{n+j}%
%TCIMACRO{\QATOPD{[}{]}{n}{j}}%
%BeginExpansion
\genfrac{[}{]}{0pt}{}{n}{j}%
%EndExpansion
_{\gamma}%
%TCIMACRO{\QATOPD{\{}{\}}{j}{m}}%
%BeginExpansion
\genfrac{\{}{\}}{0pt}{}{j}{m}%
%EndExpansion
_{\gamma}=\delta_{n,m}\hspace{30pt}(n,m\in\mathbb{N}_{0}),
\]%
\[
\sum_{m\leq j\leq n}(-1)^{j+m}%
%TCIMACRO{\QATOPD{\{}{\}}{n}{j}}%
%BeginExpansion
\genfrac{\{}{\}}{0pt}{}{n}{j}%
%EndExpansion
_{\gamma}%
%TCIMACRO{\QATOPD{[}{]}{j}{m}}%
%BeginExpansion
\genfrac{[}{]}{0pt}{}{j}{m}%
%EndExpansion
_{\gamma}=\delta_{n,m}\hspace{30pt}(n,m\in\mathbb{N}_{0}).
\]
The table below lists the Jacobi-Stirling numbers of the first kind for small
$n$ and $j$; see also Table 2 in \cite{Gelineau-Zeng}.

{\scriptsize
\[%
\begin{tabular}
[c]{|l|l|l|l|l|l|l|}\hline
$j/n$ & $n=0$ & $n=1$ & $n=2$ & $n=3$ & $n=4$ & $n=5$\\\hline
$j=0$ & $1$ & $0$ & $0$ & $0$ & $0$ & $0$\\\hline
$j=1$ & $0$ & $1$ & $2\gamma$ & $8\gamma^{2}+4\gamma$ & $48\gamma^{3}%
+72\gamma^{2}+24\gamma$ & $384\gamma^{4}+1152\gamma^{3}+1056\gamma
^{2}+288\gamma$\\\hline
$j=2$ & $0$ & $0$ & $1$ & $6\gamma+2$ & $44\gamma^{2}+52\gamma+12$ &
$400\gamma^{3}+1016\gamma^{2}+744\gamma+144$\\\hline
$j=3$ & $0$ & $0$ & $0$ & $1$ & $12\gamma+8$ & $140\gamma^{2}+260\gamma
+108$\\\hline
$j=4$ & $0$ & $0$ & $0$ & $0$ & $1$ & $20\gamma+20$\\\hline
$j=5$ & $0$ & $0$ & $0$ & $0$ & $0$ & $1$\\\hline
\end{tabular}
\ \ \ \ \
\]
}

\begin{center}
\textbf{Table 3: Jacobi-Stirling Numbers of the First Kind}
\end{center}

Like the Stirling numbers of the second kind, the signless Stirling numbers of
the first kind satisfy a triangular recurrence relation (see \cite[p.
214]{Comtet}): for all $n,j\in\mathbb{N}$ we have
\[%
%TCIMACRO{\QATOPD{[}{]}{n}{j}}%
%BeginExpansion
\genfrac{[}{]}{0pt}{}{n}{j}%
%EndExpansion
=%
%TCIMACRO{\QATOPD{[}{]}{n-1}{j-1}}%
%BeginExpansion
\genfrac{[}{]}{0pt}{}{n-1}{j-1}%
%EndExpansion
+(n-1)%
%TCIMACRO{\QATOPD{[}{]}{n-1}{j}}%
%BeginExpansion
\genfrac{[}{]}{0pt}{}{n-1}{j}%
%EndExpansion
.
\]
As we show next, the Jacobi-Stirling numbers of the first kind satisfy a
similar triangular recurrence relation.

\begin{theorem}
For all $n,j\in\mathbb{N}_{0}$ we have
\begin{equation}%
%TCIMACRO{\QATOPD{[}{]}{n}{0}}%
%BeginExpansion
\genfrac{[}{]}{0pt}{}{n}{0}%
%EndExpansion
_{\gamma}=\delta_{n,0}\hspace{20pt}\text{and}\hspace{20pt}%
%TCIMACRO{\QATOPD{[}{]}{0}{j}}%
%BeginExpansion
\genfrac{[}{]}{0pt}{}{0}{j}%
%EndExpansion
_{\gamma}=\delta_{j,0}, \label{eqn:JS1initial}%
\end{equation}
and for all $n,j\in\mathbb{N}$ we have
\begin{equation}%
%TCIMACRO{\QATOPD{[}{]}{n}{j}}%
%BeginExpansion
\genfrac{[}{]}{0pt}{}{n}{j}%
%EndExpansion
_{\gamma}=%
%TCIMACRO{\QATOPD{[}{]}{n-1}{j-1}}%
%BeginExpansion
\genfrac{[}{]}{0pt}{}{n-1}{j-1}%
%EndExpansion
_{\gamma}+(n-1)(n+2\gamma-2)%
%TCIMACRO{\QATOPD{[}{]}{n-1}{j}}%
%BeginExpansion
\genfrac{[}{]}{0pt}{}{n-1}{j}%
%EndExpansion
_{\gamma}. \label{eqn:JS1recurrence}%
\end{equation}

\end{theorem}

\begin{proof}
Line (\ref{eqn:JS1initial}) is immediate from (\ref{eqn:JS1 poly}), since
$\langle x \rangle_{0}^{(\gamma)} = 1$ and for all $j \in\mathbb{N}$ the
polynomial $\langle x \rangle_{j}^{(\gamma)}$ is divisible by $x$.

To obtain (\ref{eqn:JS1recurrence}), first note that the coefficient of
$x^{j}$ on the right side of (\ref{eqn:JS1 poly}) is $(-1)^{j+n}
%TCIMACRO{\QATOPD{[}{]}{n}{j}}%
%BeginExpansion
\genfrac{[}{]}{0pt}{}{n}{j}%
%EndExpansion
_{\gamma}$. Now observe that if $n \in\mathbb{N}$ then
\begin{align*}
\langle x \rangle_{n}^{(\gamma)}  &  = (x-(n-1)(n+2\gamma-2)) \langle x
\rangle_{n-1}^{(\gamma)}\\
&  = (x - (n-1)(n+2\gamma-2)) \sum_{j=0}^{n-1} (-1)^{n-1+j}
%TCIMACRO{\QATOPD{[}{]}{n-1}{j}}%
%BeginExpansion
\genfrac{[}{]}{0pt}{}{n-1}{j}%
%EndExpansion
_{\gamma}x^{j},
\end{align*}
so the coefficient of $x^{j}$ on the left side of (\ref{eqn:JS1 poly}) is
\[
(-1)^{j+n} (n-1)(n+2\gamma-2)%
%TCIMACRO{\QATOPD{[}{]}{n-1}{j}}%
%BeginExpansion
\genfrac{[}{]}{0pt}{}{n-1}{j}%
%EndExpansion
_{\gamma}+ (-1)^{j+n}
%TCIMACRO{\QATOPD{[}{]}{n-1}{j-1}}%
%BeginExpansion
\genfrac{[}{]}{0pt}{}{n-1}{j-1}%
%EndExpansion
_{\gamma}.
\]
We obtain (\ref{eqn:JS1recurrence}) when we equate these two expressions for
the coefficient of $x^{j}$.
\end{proof}

Next in this section, we prove a reciprocity result which connects
Jacobi-Stirling numbers of the two kinds. To state this result, first observe
that there is a unique collection $%
%TCIMACRO{\QATOPD{[}{]}{n}{j}}%
%BeginExpansion
\genfrac{[}{]}{0pt}{}{n}{j}%
%EndExpansion
_{\gamma}\ (n,j\in\mathbb{Z)}$ of polynomials in $\gamma$ satisfying the
initial condition
\begin{equation}%
%TCIMACRO{\QATOPD{[}{]}{n}{0}}%
%BeginExpansion
\genfrac{[}{]}{0pt}{}{n}{0}%
%EndExpansion
_{\gamma}=\delta_{n,0},\hspace{50pt}%
%TCIMACRO{\QATOPD{[}{]}{0}{j}}%
%BeginExpansion
\genfrac{[}{]}{0pt}{}{0}{j}%
%EndExpansion
_{\gamma}=\delta_{j,0}, \label{eqn:initial1}%
\end{equation}
and recurrence relation
\begin{equation}%
%TCIMACRO{\QATOPD{[}{]}{n}{j}}%
%BeginExpansion
\genfrac{[}{]}{0pt}{}{n}{j}%
%EndExpansion
_{\gamma}=%
%TCIMACRO{\QATOPD{[}{]}{n-1}{j-1}}%
%BeginExpansion
\genfrac{[}{]}{0pt}{}{n-1}{j-1}%
%EndExpansion
_{\gamma}+(n-1)(n+2\gamma-2)%
%TCIMACRO{\QATOPD{[}{]}{n-1}{j}}%
%BeginExpansion
\genfrac{[}{]}{0pt}{}{n-1}{j}%
%EndExpansion
_{\gamma},\hspace{30pt}(n,k\in{\mathbb{Z}}). \label{eqn:recurrence1}%
\end{equation}
Moreover, these polynomials are the Jacobi-Stirling numbers of the first kind
when $n,j\in\mathbb{N}_{0}$. Similarly, there is a unique collection $%
%TCIMACRO{\QATOPD{\{}{\}}{n}{j}}%
%BeginExpansion
\genfrac{\{}{\}}{0pt}{}{n}{j}%
%EndExpansion
_{\gamma}\ (n,j\in\mathbb{Z)}$ of polynomials in $\gamma$ satisfying the
initial condition
\begin{equation}%
%TCIMACRO{\QATOPD{\{}{\}}{n}{0}}%
%BeginExpansion
\genfrac{\{}{\}}{0pt}{}{n}{0}%
%EndExpansion
_{\gamma}=\delta_{n,0},\hspace{50pt}%
%TCIMACRO{\QATOPD{\{}{\}}{0}{j}}%
%BeginExpansion
\genfrac{\{}{\}}{0pt}{}{0}{j}%
%EndExpansion
_{\gamma}=\delta_{j,0}, \label{eqn:initial2}%
\end{equation}
and recurrence relation
\begin{equation}%
%TCIMACRO{\QATOPD{\{}{\}}{n}{j}}%
%BeginExpansion
\genfrac{\{}{\}}{0pt}{}{n}{j}%
%EndExpansion
_{\gamma}=%
%TCIMACRO{\QATOPD{\{}{\}}{n-1}{j-1}}%
%BeginExpansion
\genfrac{\{}{\}}{0pt}{}{n-1}{j-1}%
%EndExpansion
_{\gamma}+j(j+2\gamma-1)%
%TCIMACRO{\QATOPD{\{}{\}}{n-1}{j}}%
%BeginExpansion
\genfrac{\{}{\}}{0pt}{}{n-1}{j}%
%EndExpansion
_{\gamma},\hspace{30pt}(n,k\in{\mathbb{Z}}). \label{eqn:recurrence2}%
\end{equation}
Moreover, these polynomials are the Jacobi-Stirling numbers of the second kind
when $n,j\in\mathbb{N}_{0}$. It is not difficult to show that if $n\neq0$ and
$j\neq0$ differ in sign then $%
%TCIMACRO{\QATOPD{[}{]}{n}{j}}%
%BeginExpansion
\genfrac{[}{]}{0pt}{}{n}{j}%
%EndExpansion
_{\gamma}=%
%TCIMACRO{\QATOPD{\{}{\}}{n}{j}}%
%BeginExpansion
\genfrac{\{}{\}}{0pt}{}{n}{j}%
%EndExpansion
_{\gamma}=0$, but we might hope that for various negative $n$ and $j$ we
obtain interesting new polynomials $%
%TCIMACRO{\QATOPD{[}{]}{n}{j}}%
%BeginExpansion
\genfrac{[}{]}{0pt}{}{n}{j}%
%EndExpansion
_{\gamma}$ and $%
%TCIMACRO{\QATOPD{\{}{\}}{n}{j}}%
%BeginExpansion
\genfrac{\{}{\}}{0pt}{}{n}{j}%
%EndExpansion
_{\gamma}$. Our reciprocity result, which is an analogue of a similar result
\cite[Line (2.4)]{Knuth} for classical Stirling numbers, shows that we
actually obtain two familiar families of polynomials.

\begin{theorem}
For all $n,j \in\mathbb{Z}$ we have
\begin{equation}
\label{eqn:JSrecip}%
%TCIMACRO{\QATOPD{\{}{\}}{-j}{-n}}%
%BeginExpansion
\genfrac{\{}{\}}{0pt}{}{-j}{-n}%
%EndExpansion
_{\gamma}= (-1)^{n+j}
%TCIMACRO{\QATOPD{[}{]}{n}{j}}%
%BeginExpansion
\genfrac{[}{]}{0pt}{}{n}{j}%
%EndExpansion
_{1-\gamma}%
\end{equation}

\end{theorem}

\begin{proof}
The Jacobi-Stirling numbers of the second kind are uniquely determined by
(\ref{eqn:initial2}) and (\ref{eqn:recurrence2}), so it is sufficient to show
that the quantities $L(n,j) = (-1)^{n+j}
%TCIMACRO{\QATOPD{[}{]}{-j}{-n}}%
%BeginExpansion
\genfrac{[}{]}{0pt}{}{-j}{-n}%
%EndExpansion
_{1-\gamma}$ also satisfy (\ref{eqn:initial2}) and (\ref{eqn:recurrence2}).

The fact that $L(n,j)$ satisfies (\ref{eqn:initial2}) is immediate from
(\ref{eqn:initial1}), so suppose $n\neq0$ and $j\neq0$. Then we have
\begin{align*}
L(n-1,j-1)  &  =(-1)^{n+j}%
%TCIMACRO{\QATOPD{[}{]}{-j+1}{-n+1}}%
%BeginExpansion
\genfrac{[}{]}{0pt}{}{-j+1}{-n+1}%
%EndExpansion
_{1-\gamma}\\
&  =(-1)^{n+j}%
%TCIMACRO{\QATOPD{[}{]}{-j}{-n}}%
%BeginExpansion
\genfrac{[}{]}{0pt}{}{-j}{-n}%
%EndExpansion
_{1-\gamma}+(-1)^{n+j}(-j)(-j+1+2(1-\gamma)-2)%
%TCIMACRO{\QATOPD{[}{]}{-j}{-n+1}}%
%BeginExpansion
\genfrac{[}{]}{0pt}{}{-j}{-n+1}%
%EndExpansion
_{1-\gamma}\\
&  =L(n,j)-j(j+2\gamma-1)L(n-1,j),
\end{align*}
and the result follows.
\end{proof}

Concluding this section, we now turn our attention to the unimodality of the
Jacobi-Stirling numbers of the first and second kinds. Recall that a sequence
of real numbers $\{x_{n}\}_{n=0}^{\infty}$ is \textit{unimodal }if there
exists two integers\textit{ }$N_{1}\leq N_{2}$ such that (i) $k\leq N_{1}-2$
implies $a_{k}\leq a_{k+1},$ (ii) $a_{N_{1}-1}<a_{N_{1}}=\cdots=a_{N_{2}%
}>a_{N_{2}+1},$ and (iii) $k\geq N_{2}+1$ implies $a_{k}\geq a_{k+1}.$

As in \cite[Section 5.7]{Andrews-Gawronski-Littlejohn}, we consider the
horizontal generating function%
\begin{equation}
A_{n}(x)=\sum_{j=0}^{n}%
%TCIMACRO{\QATOPD{\{}{\}}{n}{j}}%
%BeginExpansion
\genfrac{\{}{\}}{0pt}{}{n}{j}%
%EndExpansion
_{\gamma}x^{j}\quad(n\in\mathbb{N}_{0}) \label{A_n(x)}%
\end{equation}
which, by Theorem \ref{Properties of the Jacobi-Stirling numbers} (iii), is a
polynomial of degree $n$ with leading coefficient $%
%TCIMACRO{\QATOPD{\{}{\}}{n}{n}}%
%BeginExpansion
\genfrac{\{}{\}}{0pt}{}{n}{n}%
%EndExpansion
_{\gamma}=1.$ Again, by this property, we find that $A_{0}(x)=1$ and, for
$n\geq1,$%
\begin{align*}
A_{n}(x)  &  =\sum_{j=1}^{n}%
%TCIMACRO{\QATOPD{\{}{\}}{n-1}{j-1}}%
%BeginExpansion
\genfrac{\{}{\}}{0pt}{}{n-1}{j-1}%
%EndExpansion
_{\gamma}x^{j}+\sum_{j=1}^{n}j(j+2\gamma-1)%
%TCIMACRO{\QATOPD{\{}{\}}{n-1}{j}}%
%BeginExpansion
\genfrac{\{}{\}}{0pt}{}{n-1}{j}%
%EndExpansion
_{\gamma}x^{j}\\
&  =x(A_{n-1}(x)+2\gamma A_{n-1}^{\prime}(x)+xA_{n-1}^{\prime\prime}(x)).
\end{align*}
For example, we see that%
\[
A_{1}(x)=x,\;A_{2}(x)=x(x+2\gamma),\;A_{3}(x)=x(x^{2}+2(3\gamma+1)x+4\gamma
^{2}).
\]
Since $\gamma>0,$ $A_{1},$ $A_{2},$ and $A_{3}$ are real polynomials, the
zeros of which are all real, simple, and non-positive. Along the same lines as
the proof of Theorem 5.7 in \cite{Andrews-Gawronski-Littlejohn}, we obtain

\begin{theorem}
\label{Theorem 5.3}Let $n\in\mathbb{N}.$ The zeros of $A_{n}$, defined in
$($\ref{A_n(x)}$),$ are real, simple and non-positive. Moreover, $A_{n}(0)=0.$
\end{theorem}

For a different proof, see \cite[Theorem 4]{Mongelli1}. From Theorem
\ref{Theorem 5.3}, equation (\ref{eqn:JS1 poly}), and a standard criterion for
unimodality as given in Comtet \cite[p. 270]{Comtet}, we can state the
following result.

\begin{theorem}
The unsigned Jacobi-Stirling numbers of the first kind and the Jacobi-Stirling
numbers of the second kind are unimodal with either a peak or a plateau of $2$ points.
\end{theorem}

\section{Two Combinatorial Interpretations of the Jacobi-Stirling Numbers of
the First Kind\label{sec:JS1combin}}

The signless classical Stirling number of the first kind $%
%TCIMACRO{\QATOPD{[}{]}{n}{j}}%
%BeginExpansion
\genfrac{[}{]}{0pt}{}{n}{j}%
%EndExpansion
$ is the number of permutations of $\{1,2,\ldots,n\}$ with $j$ cycles. With
this in mind, it is natural to ask for a similar combinatorial interpretation
of the Jacobi-Stirling number of the first kind $%
%TCIMACRO{\QATOPD{[}{]}{n}{j}}%
%BeginExpansion
\genfrac{[}{]}{0pt}{}{n}{j}%
%EndExpansion
_{\gamma}$. Indeed, Egge \cite{Egge} has given a combinatorial interpretation
of the Legendre-Stirling number $%
%TCIMACRO{\QATOPD{[}{]}{n}{j}}%
%BeginExpansion
\genfrac{[}{]}{0pt}{}{n}{j}%
%EndExpansion
_{1}$ in terms of pairs of permutations; one might call these
\emph{Legendre-Stirling permutation pairs}. More recently, Gelineau and Zeng
\cite{Gelineau-Zeng} have found a statistic on Legendre-Stirling permutation
pairs which allows them to interpret $%
%TCIMACRO{\QATOPD{[}{]}{n}{j}}%
%BeginExpansion
\genfrac{[}{]}{0pt}{}{n}{j}%
%EndExpansion
_{\gamma}$ as a generating function over these pairs. In this section we
generalize Legendre-Stirling permutation pairs still further, to obtain
objects we will call \emph{Jacobi-Stirling permutation pairs}. These objects
will come in two flavors, balanced and unbalanced. When $\gamma=1$ it will be
clear that both the balanced and the unbalanced Jacobi-Stirling permutation
pairs are in fact Legendre-Stirling permutation pairs. In addition, we will
show that for any positive integer $\gamma$, the Jacobi-Stirling numbers count
both the balanced and the unbalanced Jacobi-Stirling permutation pairs. In
connection with these pairs, it will be useful to note that the \emph{cycle
maxima} of a given permutation are the numbers which are largest in their
cycles. For example, if $\pi=(4,6,1)(9,2,3)(7,8)$ is a permutation in $S_{10}%
$, written in cycle notation, then its cycle maxima are $5,6,8,9,$ and $10$.

\begin{definition}
\label{defn:JSpp} Suppose $n,\gamma\in\mathbb{N}_{0}$. A \emph{balanced
Jacobi-Stirling permutation pair} of length $n$ is an ordered pair $(\pi
_{1},\pi_{2})$ with $\pi_{1}\in S_{n+\gamma}$ and $\pi_{2}\in S_{n+\gamma-1}$
for which the following conditions hold:

\begin{enumerate}
\item $\pi_{1}$ has one more cycle than $\pi_{2}$.

\item The cycle maxima of $\pi_{1}$ which are less than $n + \gamma$ are
exactly the cycle maxima of $\pi_{2}$.

\item For each $k$ which is not a cycle maximum, at least one of $\pi_{1}(k)$
and $\pi_{2}(k)$ is less than or equal to $n$.
\end{enumerate}
\end{definition}

In \cite{Egge} Egge defines a Legendre-Stirling permutation pair of length $n$
to be an ordered pair $(\pi_{1},\pi_{2})$ with $\pi_{1} \in S_{n+1}$ and
$\pi_{2} \in S_{n}$, such that $\pi_{1}$ has one more cycle than $\pi_{2}$ and
the cycle maxima of $\pi_{1}$ which are less than $n+1$ are exactly the cycle
maxima of $\pi_{2}$. If $(\pi_{1},\pi_{2})$ is an ordered pair of permutations
with $\pi_{1} \in S_{n+1}$ and $\pi_{2} \in S_{n}$ which satisfies conditions
(1) and (2) of Definition \ref{defn:JSpp}, then the fact that $\pi_{2} \in
S_{n}$ implies $(\pi_{1},\pi_{2})$ also satisfies condition (3). Therefore,
the Legendre-Stirling permutation pairs of length $n$ are exactly the balanced
Jacobi-Stirling permutation pairs of length $n$ with $\gamma= 1$. The
Legendre-Stirling permutation pairs are counted by the Legendre-Stirling
numbers of the first kind $%
%TCIMACRO{\QATOPD{[}{]}{n}{j}}%
%BeginExpansion
\genfrac{[}{]}{0pt}{}{n}{j}%
%EndExpansion
_{1}$; as we show next, the balanced Jacobi-Stirling permutation pairs are
counted by the Jacobi-Stirling numbers of the first kind.

\begin{theorem}
\label{thm:balancedJSPPcount} For all $n, j, \gamma\in\mathbb{N}_{0}$, the
number of balanced Jacobi-Stirling permutation pairs $(\pi_{1},\pi_{2})$ of
length $n$ in which $\pi_{1}$ has exactly $\gamma+j$ cycles is $%
%TCIMACRO{\QATOPD{[}{]}{n}{j}}%
%BeginExpansion
\genfrac{[}{]}{0pt}{}{n}{j}%
%EndExpansion
_{\gamma}$.
\end{theorem}

\begin{proof}
Let $a_{n,j}$ denote the number of balanced Jacobi-Stirling permutation pairs
$(\pi_{1},\pi_{2})$ of length $n$ in which $\pi_{1}$ has exactly $\gamma+j$ cycles.

It follows from (\ref{eqn:initial1}) and (\ref{eqn:recurrence1}) that $%
%TCIMACRO{\QATOPD{[}{]}{1}{j}}%
%BeginExpansion
\genfrac{[}{]}{0pt}{}{1}{j}%
%EndExpansion
_{\gamma}= \delta_{j,1}$. On the other hand, if $(\pi_{1},\pi_{2})$ is a
balanced Jacobi-Stirling permutation pair of length $1$ in which $\pi_{1}$ has
exactly $\gamma+j$ cycles, then we must have $j \le1$, since $\pi_{1} \in
S_{1+\gamma}$. If $j = 0$ then some entry of $\pi_{1}$ violates Definition
\ref{defn:JSpp}(3), so we must have $j = 1$. Moreover, when $j = 1$ both
$\pi_{1}$ and $\pi_{2}$ must be the identity permutation, so $a_{1,j} =
\delta_{j,1}$. Therefore the result holds for $n = 1$.

Now suppose $n > 1$ and the result holds for $n - 1$; we argue by induction on
$n$. To obtain $a_{n,j}$, first observe that by condition (2) of Definition
\ref{defn:JSpp}, if $(\pi_{1},\pi_{2})$ is a balanced Jacobi-Stirling
permutation pair of length $n$ then 1 is a fixed point in $\pi_{1}$ if and
only if it is a fixed point in $\pi_{2}$. Pairs $(\pi_{1},\pi_{2})$ in which 1
is a fixed point are in bijection with pairs $(\sigma_{1},\sigma_{2})$ of
length $n-1$ in which $\sigma_{1}$ has $j-1+\gamma$ cycles by removing the 1
from each permutation and decreasing all other entries by 1. Each pair
$(\pi_{1},\pi_{2})$ in which 1 is not a fixed point may be constructed
uniquely by choosing a pair $(\sigma_{1},\sigma_{2})$ of length $n-1$ in which
$\sigma_{1}$ has $j+\gamma$ cycles, increasing each entry of each permutation
by 1, and inserting 1 after an entry of each permutation. There are
$a_{n-1,j}$ pairs $(\sigma_{1},\sigma_{2})$, there are $(n-1)(n+\gamma-2)$
ways to insert 1 so that $\pi_{1}(1) \le n$, and there are $\gamma(n-1)$ ways
to insert 1 so that $\pi_{1}(1) > n$. Combining these observations and using
induction to eliminate $a_{n-1,j-1}$ and $a_{n-1,k}$ we find
\begin{align*}
a_{n,j}  &  =
%TCIMACRO{\QATOPD{[}{]}{n-1}{j-1}}%
%BeginExpansion
\genfrac{[}{]}{0pt}{}{n-1}{j-1}%
%EndExpansion
_{\gamma}+ (n-1)(n+2\gamma-2)%
%TCIMACRO{\QATOPD{[}{]}{n-1}{j}}%
%BeginExpansion
\genfrac{[}{]}{0pt}{}{n-1}{j}%
%EndExpansion
_{\gamma}\\
&  =
%TCIMACRO{\QATOPD{[}{]}{n}{j}}%
%BeginExpansion
\genfrac{[}{]}{0pt}{}{n}{j}%
%EndExpansion
_{\gamma},
\end{align*}
as desired.
\end{proof}

We have shown that the Jacobi-Stirling numbers of the first kind count
balanced Jacobi-Stirling permutation pairs whenever $\gamma$ is an integer,
but the orthogonal polynomials that give rise to these numbers include
interesting special cases in which $\gamma$ is a half integer. Most notably,
the Tchebycheff polynomials of the first kind occur when $\gamma= \frac12$,
while the Tchebycheff polynomials of the second kind occur when $\gamma=
\frac32$. To address the combinatorics of the case in which $2\gamma$ is an
integer, we introduce unbalanced Jacobi-Stirling permutation pairs, and we
show that these pairs are also counted by the Jacobi-Stirling numbers of the
first kind.

\begin{definition}
\label{defn:unbalancedJSpp} Suppose $n,2\gamma\in\mathbb{N}_{0}$. An
\emph{unbalanced Jacobi-Stirling permutation pair} of length $n$ is an ordered
pair $(\pi_{1},\pi_{2})$ with $\pi_{1} \in S_{n+2\gamma-1}$ and $\pi_{2} \in
S_{n}$ for which the following hold.

\begin{enumerate}
\item $\pi_{1}$ has $2\gamma-1$ more cycles than $\pi_{2}$.

\item The cycle maxima of $\pi_{1}$ which are less than $n + 1$ are exactly
the cycle maxima of $\pi_{2}$.
\end{enumerate}
\end{definition}

It's not difficult to see that when $\gamma= 1$ the unbalanced Jacobi-Stirling
permutation pairs are exactly the Legendre-Stirling permutation pairs
introduced by Egge, which are counted by the Legendre-Stirling numbers of the
first kind $%
%TCIMACRO{\QATOPD{[}{]}{n}{j}}%
%BeginExpansion
\genfrac{[}{]}{0pt}{}{n}{j}%
%EndExpansion
_{1}$. As we show next, the unbalanced Jacobi-Stirling permutation pairs are
counted by the Jacobi-Stirling numbers of the first kind. We note that this
result is a special case of a result proved independently by Mongelli
\cite{Mongelli2}.

\begin{theorem}
\label{thm:JS1unbalanced} For all $n, j, \gamma\in\mathbb{N}_{0}$, the number
of unbalanced Jacobi-Stirling permutation pairs $(\pi_{1},\pi_{2})$ of length
$n$ in which $\pi_{2}$ has exactly $j$ cycles is $%
%TCIMACRO{\QATOPD{[}{]}{n}{j}}%
%BeginExpansion
\genfrac{[}{]}{0pt}{}{n}{j}%
%EndExpansion
_{\gamma}$.
\end{theorem}

\begin{proof}
Let $a_{n,j}$ denote the number of unbalanced Jacobi-Stirling permutation
pairs $(\pi_{1},\pi_{2})$ of length $n$ in which $\pi_{2}$ has exactly $j$ cycles.

It follows from (\ref{eqn:initial1}) and (\ref{eqn:recurrence1}) that $%
%TCIMACRO{\QATOPD{[}{]}{1}{j}}%
%BeginExpansion
\genfrac{[}{]}{0pt}{}{1}{j}%
%EndExpansion
_{\gamma}= \delta_{j,1}$. On the other hand, if $(\pi_{1},\pi_{2})$ is an
unbalanced Jacobi-Stirling permutation pair of length $1$ in which $\pi_{2}$
has exactly $j$ cycles, then we must have $j = 1$, since $\pi_{2} \in S_{1}$.
Moreover, when $j = 1$ both $\pi_{1}$ and $\pi_{2}$ must be the identity
permutation, so $a_{1,j} = \delta_{j,1}$. Therefore the result holds for $n =
1$.

Now suppose $n > 1$ and the result holds for $n - 1$; we argue by induction on
$n$. To obtain $a_{n,j}$, first observe that by condition 2 of Definition
\ref{defn:unbalancedJSpp}, if $(\pi_{1},\pi_{2})$ is an unbalanced
Jacobi-Stirling permutation pair of length $n$ then 1 is a fixed point in
$\pi_{1}$ if and only if it is a fixed point in $\pi_{2}$. Pairs $(\pi_{1}%
,\pi_{2})$ in which 1 is a fixed point are in bijection with pairs
$(\sigma_{1},\sigma_{2})$ of length $n-1$ in which $\sigma_{2}$ has $j-1$
cycles by removing the 1 from each permutation and decreasing all other
entries by 1. Each pair $(\pi_{1},\pi_{2})$ in which 1 is not a fixed point
may be constructed uniquely by choosing a pair $(\sigma_{1},\sigma_{2})$ of
length $n-1$ in which $\sigma_{2}$ has $j$ cycles, increasing each entry of
each permutation by 1, and inserting 1 after an entry of each permutation.
There are $a_{n-1,j}$ pairs $(\sigma_{1},\sigma_{2})$, and there are
$(n-1)(n+2\gamma-2)$ ways to insert our 1s. Combining these observations and
using induction to eliminate $a_{n-1,j-1}$ and $a_{n-1,k}$ we find
\begin{align*}
a_{n,j}  &  =
%TCIMACRO{\QATOPD{[}{]}{n-1}{j-1}}%
%BeginExpansion
\genfrac{[}{]}{0pt}{}{n-1}{j-1}%
%EndExpansion
_{\gamma}+ (n-1)(n+2\gamma-2)%
%TCIMACRO{\QATOPD{[}{]}{n-1}{j}}%
%BeginExpansion
\genfrac{[}{]}{0pt}{}{n-1}{j}%
%EndExpansion
_{\gamma}\\
&  =
%TCIMACRO{\QATOPD{[}{]}{n}{j}}%
%BeginExpansion
\genfrac{[}{]}{0pt}{}{n}{j}%
%EndExpansion
_{\gamma},
\end{align*}
as desired.
\end{proof}

We conclude with an example involving unbalanced Jacobi-Stirling permutation pairs.

\begin{example}
In this example we give a direct combinatorial proof that%
\[
{%
%TCIMACRO{\QATOPD{[}{]}{n}{1}}%
%BeginExpansion
\genfrac{[}{]}{0pt}{}{n}{1}%
%EndExpansion
_{\gamma}=(n-1)!\prod_{j=0}^{n-1}(2\gamma+j)}.
\]
By Theorem \ref{thm:JS1unbalanced}, the quantity $%
%TCIMACRO{\QATOPD{[}{]}{n}{1}}%
%BeginExpansion
\genfrac{[}{]}{0pt}{}{n}{1}%
%EndExpansion
_{\gamma}$ is the number of unbalanced Jacobi-Stirling permutation pairs
$(\pi_{1},\pi_{2})$ of length $n$, where $\pi_{1}\in S_{n+2\gamma-1}$ has
$2\gamma$ cycles and $\pi_{2}\in S_{n}$ has 1 cycle. Since $\pi_{2}$ has 1
cycle, there are $(n-1)!$ choices for this permutation. Moreover, the cycle
maxima for $\pi_{1}$ must be $n,n+1,\ldots,n+2\gamma-1$. Now we may place 1 in
$\pi_{2}$ in $2\gamma$ ways, we may place 2 in $2\gamma+1$ ways, and in
general we may place $j+1$ in $2\gamma+j$ ways. Thus there are $(n-1)!\prod
_{j=0}^{n-1}(2\gamma+j)$ of these unbalanced Jacobi-Stirling permutation pairs.
\end{example}

\end{document}